\documentclass[10pt]{amsart}
\usepackage{amsmath, amsfonts, amscd, latexsym, amsthm, amssymb}

\def \c{\mathbb{C}}
\def \z{\mathbb{Z}}
\def \r{\mathbb{R}}
\def \n{\mathbb{N}}
\def \p{\mathbb{P}}

\def \K{\mathcal{K}}

\def \SL{\textup{SL}}

\def \SP{\textup{SP}}
\def \SO{\textup{SO}}

\def \.{\cdot}
\def \Con{\textup{Con}}

\def \dim{\textup{dim}}

\def \Vol{\textup{Vol}}

\def \supp{\textup{supp}}
\def \Area{\textup{Area}}

\def \ord{\textup{ord}}

\def \K{{\bf K}_{rat}(X)}
\def \qp{quasi-projective }

\theoremstyle{plain}
\newtheorem{Th}{Theorem}[section]

\newtheorem{Prop}[Th]{Proposition}
\newtheorem{Cor}[Th]{Corollary}

\theoremstyle{definition}
\newtheorem{Ex}[Th]{Example}
\newtheorem{Def}[Th]{Definition}

\begin{document}
\title{Algebraic equations and convex bodies}
\author{Kiumars Kaveh, A. G. Khovanskii}
\maketitle

\begin{center}
{\it Dedicated to Oleg Yanovich Viro on the occasion of his sixtieth birthday}
\end{center}
\vspace{.4cm}

{\small \noindent{\bf Abstract.} The well-known Bernstein-Ku\v{s}hnirenko
theorem from the theory of Newton polyhedra relates algebraic geometry
and the theory of mixed volumes. Recently the authors have found a
far-reaching generalization of this theorem to generic systems of
algebraic equations on any quasi-projective variety. In the present
note we review these results and their applications to algebraic
geometry and convex geometry.}\\



\tableofcontents

\section{Introduction}
The famous Bernstein-Ku\v{s}hnirenko theorem from the Newton
polyhedra theory relates algebraic geometry (mainly the theory of
toric varieties) with the theory of mixed volumes in convex geometry.
This relation is
useful in both directions. On one hand it allows one to prove
Alexandrov-Fenchel inequality (the most important and the hardest
result in the theory of mixed volumes) using Hodge inequality
from the theory of algebraic surfaces. On the other hand it suggests
new inequalities in the intersection theory of Cartier divisors
analogous to the known inequalities for mixed volumes.

Recently the authors have found a far-reaching generalization of Ku\v{s}hnirenko
theorem in which instead of complex torus $(\c^*)^n$ we consider any
quasi-projective variety $X$ and instead of a finite dimensional space of functions
spanned by monomials in $(\c^*)^n$, we consider any finite
dimensional space of rational functions on $X$.

To this end, firstly we develop an intersection theory for finite dimensional subspaces of rational functions
on a \qp variety. It can be considered as a generalization of the
intersection theory of Cartier divisors for a (non -complete) variety $X$.
We show that this intersection theory enjoys all the properties of mixed volumes
\cite{Askold-Kiumars-arXiv2}. Secondly, we introduce the {\it Newton convex body} which is a
far generalization of the Newton polyhedron of a Laurent polynomial.
Our construction of Newton convex body depends on a fixed $\z^n$-valued valuation on
the field of rational functions on $X$. It associates to any finite dimensional space $L$
of rational functions on $X$, its Newton convex body $\Delta(L)$. We
obtain a direct generalization of the Ku\v{s}nirenko theorem in
this setting (see Theorem \ref{th-main}).

This construction then allows us to give a proof of the Hodge
inequality using elementary geometry of planar convex domains and
(as a corollary) an elementary proof of Alexandrov-Fenchel
inequality. In general our construction does not imply a
generalization of the Bernstein theorem. Although we also obtain a generalization of
this theorem for some cases when the variety $X$ is equipped with a
reductive group action.

In this paper we present a review of the results mentioned above. We omit the proofs in this
short note. A preliminary version together with proofs can be found at \cite{Askold-Kiumars-arXiv}.
The paper \cite{Askold-Kiumars-arXiv2} contains a detailed version of the first half of the preprint
\cite{Askold-Kiumars-arXiv}, and a detailed version of the second half of \cite{Askold-Kiumars-arXiv}
will appear very soon.

After posting of these results in the arXiv, we learned that we were not the only ones
to have been working in this direction. Firstly, A. Okounkov was the pioneer to define (in
passing) an analogue of Newton polytope in general situation
in his interesting papers \cite{Okounkov-Brunn-Minkowski, Okounkov-log-concave}
(although his case of interest was when $X$ has a reductive group action).
Secondly, R. Lazarsfeld and M. Mustata, based on Okounkov's previous
works, and independently of our preprint, have come up with closely
related results \cite{Lazarsfeld-Mustata}. Recently, following
\cite{Lazarsfeld-Mustata}, similar results/constructions have been
obtained for line bundles on arithmetic surfaces \cite{Yuan}.

\section{Mixed volume}
By a {\it convex body} we mean a convex compact subset of $\r^n$. There are two operations of addition
and scalar multiplication for convex bodies: let $\Delta_1$, $\Delta_2$ be convex bodies, then
their sum
$$\Delta_1 + \Delta_2 = \{x+y \mid x \in \Delta_1, y \in \Delta_2 \},$$
is also a convex body called the {\it Minkowski sum} of $\Delta_1$, $\Delta_2$. Also for a convex body $\Delta$ and
a scalar $\lambda \geq 0$, $$\lambda \Delta = \{\lambda x \mid x \in \Delta \},$$ is a convex body.

Let $\Vol$ denotes the $n$-dimensional volume in $\r^n$ with respect to the
standard Euclidean metric. Function $\Vol$ is a homogeneous polynomial of degree $n$ on the cone of convex bodies,
i.e. its restriction to each finite dimensional section of the cone is a homogeneous polynomial of degree $n$. More
precisely: for any $k>0$ let $\r_+^k$ be the positive octant in $\r^k$ consisting of all
${\bf \lambda} = (\lambda_1,\dots,\lambda_k)$ with $\lambda_1\geq
0,\dots,\lambda_k\geq 0$. Then polynomiality of $\Vol$ means that
for any choice of convex bodies $\Delta_1,\dots,\Delta _k$, the function
$P_{\Delta_1,\dots,\Delta_k}$ defined on $\r^k_{+}$ by
$$P_{\Delta_1,\dots,\Delta_k}(\lambda_1,\dots,\lambda_k)=\Vol(\lambda_1\Delta_1+\dots+
\lambda_k\Delta_k),$$ is a homogeneous polynomial of degree $n$.

By definition the {\it mixed volume of}  $V(\Delta_1,\dots,\Delta_n)$ of
an $n$-tuple $(\Delta_1,\dots,\Delta_n)$ of convex bodies
is the coefficient of the monomial $\lambda_1\dots\lambda_n$ in the polynomial
$P_{\Delta_1,\dots,\Delta_n}$ divided by $n!$.

This definition implies that the mixed volume is the polarization of the volume polynomial,
that is, it is a function on the $n$-tuples of convex bodies satisfying the
following:
\begin{itemize}
\item[(i)] (Symmetry) $V$ is symmetric with respect to permuting the bodies $\Delta_1, \ldots, \Delta_n$.
\item[(ii)] (Multi-linearity) It is linear in each argument with respect to the Minkowski sum. Linearity in the first argument
means that for convex bodies $\Delta_1'$, $\Delta_1''$ and
$\Delta_2,\dots,\Delta_n$ we have:
$$ V(\Delta_1'+\Delta_1'',\dots, \Delta_n)=V(\Delta_1',\dots,
\Delta_n)+V(\Delta_1'',\dots, \Delta_n).$$

\item[(iii)] (Relation with volume) On the diagonal it coincides with volume, i.e. if
$\Delta_1 =\dots=\Delta_n=\Delta$, then $V(\Delta_1,\dots,
\Delta_n)=\Vol(\Delta)$.
\end{itemize}

The above three properties characterize the mixed volume: it is the
unique function satisfying (i)-(iii).

The following two inequalities are easy to verify:

\noindent 1) Mixed volume is non-negative, that is, for any $n$-tuple
of convex bodies $\Delta_1, \dots, \Delta_n$ we have
$$V(\Delta_1,\dots, \Delta_n)\geq 0.$$

\noindent 2) Mixed volume is monotone, that is,
for two $n$-tuples of convex bodies $\Delta'_1\subset \Delta_1,\dots, \Delta'_n\subset
\Delta_n$ we have $$V(\Delta_1,\dots, \Delta_n)\geq
V(\Delta'_1,\dots, \Delta'_n).$$

The following inequality attributed to Alexandrov and Fenchel is important and very
useful in convex geometry. All its previously known proofs are rather complicated ( see
\cite{Burago-Zalgaller}).
\begin{Th}[Alexandrov-Fenchel] \label{thm-alexandrov-fenchel}
Let $\Delta_1, \ldots, \Delta_n$ be convex bodies
in $\r^n$. Then
$$ V(\Delta_1,
\Delta_2, \ldots, \Delta_n)^2 \geq V(\Delta_1, \Delta_1, \Delta_3,
\ldots, \Delta_n) V(\Delta_2, \Delta_2, \Delta_3, \ldots, \Delta_n)
.$$
\end{Th}

Below we mention a formal corollary of Alexandrov-Fenchel inequality.
First we need to introduce a notation for when we have repetition
of convex bodies in the mixed volume. Let $2\leq m\leq n$
be an integer and $k_1+\dots+k_r=m$ a partition of $m$ with $k_i \in \n$.
Denote by $V(k_1*\Delta_1,\dots, k_r*\Delta_r,\Delta_{m+1},\dots,\Delta_n)$
the mixed volume of the $\Delta_i$ where
$\Delta_1$ is repeated $k_1$ times, $\Delta_2$ is repeated $k_2$ times, etc.
and $\Delta_{m+1},\dots,\Delta_n$ appear once.

\begin{Cor}  With the notation as above, the following inequality holds:
$$V^m(k_1*\Delta_1,\dots,
k_r*\Delta_r,\Delta_{m+1},\dots,\Delta_n)\geq \prod\limits _{1\leq j
\leq r}  V^{k_j}(m*\Delta_j,\Delta_{m+1}\dots,\Delta_n).$$
\end{Cor}

\section{Brunn-Minkowski inequality}
The celebrated {\it Brunn-Minkowski inequality} concerns volume of
convex bodies in $\r^n$.

\begin{Th}[Brunn-Minkowski] \label{th-Brunn-Mink}
Let $\Delta_1$, $\Delta_2$ be convex bodies in $\r^n$. Then
$$\Vol^{1/n}(\Delta_1) + \Vol^{1/n}(\Delta_2)\leq
\Vol^{1/n}(\Delta_1+\Delta_2).$$
\end{Th}

The inequality was first found and proved by Brunn around the end of 19th
century in the following form.
\begin{Th} \label{th-Brunn-Mink-alt-form}
Let $V_\Delta(h)$ be the $n$-dimensional volume of the section
$x_{n+1}=h$ of a convex body $\Delta\subset \r^{n+1}$. Then
$V_\Delta^{1/n}(h)$ is a  concave function in $h$.
\end{Th}

To obtain Theorem \ref{th-Brunn-Mink} from Theorem \ref{th-Brunn-Mink-alt-form},
one takes $\Delta \subset \r^{n+1}$ to be the convex combination of $\Delta_1$ and $\Delta_2$, i.e.
$$\Delta = \{(x, h) \mid 0 \leq h \leq 1,~ x \in h\Delta_1 + (1-h)\Delta_2 \}.$$
The concavity of the function $$
V_\Delta(h) = \Vol(h \Delta_1 + (1-h) \Delta_2),$$ then readily implies Theorem \ref{th-Brunn-Mink}.




For $n=2$ Theorem \ref{th-Brunn-Mink-alt-form} is equivalent to the
Alexandrov-Fenchel inequality (see Theorem \ref{th-isoper}).
Below we give a sketch of its proof in the general case.
\begin{proof}[Sketch of proof of Theorem \ref{th-Brunn-Mink-alt-form}]
1) When the convex body $\Delta \subset \r^{n+1}$ is rotationally
symmetric with respect to the $x_{n+1}$-axis, Theorem
\ref{th-Brunn-Mink-alt-form} is obvious.

2) Now suppose $\Delta$ is not rotationally symmetric.
Fix a hyperplane $H$ containing the $x_{n+1}$-axis. Then one can
construct a new convex body $\Delta'$ which is symmetric with
respect to the hyperplane $H$ and such that the volume of sections
of $\Delta'$ is the same as that of $\Delta$. To do this, just think
of $\Delta$ as the union of line segments perpendicular to the plane
$H$. Then shift each segment, along its line, in such a way that its
center lies on $H$. The resulting body is then symmetric with
respect to $H$ and has the same volume of sections as $\Delta$.
The above construction is called  the {\it Steiner symmetrization process}.

3) Consider the set of all convex bodies inside a bounded closed
domain equipped with the Hausdorff metric. One checks that this
set is compact.

4) Take the collection of all convex bodies that can be obtained
from $\Delta$ by a finite number of symmetrizations with respect to
hyperplanes $H$ containing the $x_{n+1}$-axis. The closure
$C(\Delta)$ of this collection with respect to the Hausdorff metric
is compact.

5) Take the body $\Delta_{rot}$ which is rotationally symmetric with
respect to the $x_{n+1}$-axis and with the following
property: the volume of any section of $\Delta_{rot}$ by a
horizontal hyperplane is equal to the volume of the section of
$\Delta$ by the same hyperplane. A priori we do not know that
$\Delta_{rot}$ is convex.

6) Consider the continuous function $f$ on $C(\Delta)$ defined by
$$f(\Delta_1)= V(\Delta_1\setminus \Delta_{rot})+V(
\Delta_{rot}\setminus\Delta_1),$$  for any $\Delta_1\in C(X)$.
Since $f$ is continuous it has a minimum. Take a
body $\Delta_0\in C(\Delta)$ at which $f$ attains a minimum. Let us
show $\Delta_0=\Delta_{rot}$.

7) If $\Delta_0\neq \Delta_{rot}$ then there is a horizontal
hyperplane $L$ and points $a,b\in L$ such that $a\in
\Delta_{rot}$, $a \notin \Delta_0$ and $b \notin \Delta_{rot}$, $b\in
\Delta_0$ (note that $V(\Delta_{rot}\cap L)= V(\Delta_{0}\cap L$).
Consider the line $\ell$ passing through the points $a$ and $b$. Take
the hyperplane $H$ orthogonal to $\ell$ and containing $x_{n+1}$-axis.
Denote by $\Delta_0'$ the result of the Steiner
symmetrization of $\Delta_0$ with respect to $H$.
It is easy to check that $f(\Delta_0')<f(\Delta_0)$, which
contradict the minimality of $f(\Delta_0)$. The contradiction proves
that $\Delta_0=\Delta_{rot}$ and thus $\Delta_{rot}$ is convex.
By Step 1) we then have the required inequality and the proof is finished.
\end{proof}

\section{Brunn-Minkowski and Alexandrov-Fenchel inequalities}
Let us  remind the classical {\it isoperimetric inequality} whose origins
date back to the antiquity. According to
this inequality if $P$ is the perimeter of a simple closed curve in the plane and
$A$ is the area enclosed by the curve then
\begin{equation} \label{equ-isoper}
4\pi A \leq P^2.
\end{equation}

The equality is obtained when the curve is a circle. To prove
(\ref{equ-isoper}) it is enough to prove it for convex regions. The
Alexandrov-Fenchel inequality for $n=2$ implies the isoperimetric inequality
(\ref{equ-isoper}) as a particular case and hence has inherited the name.

\begin{Th}[Isoperimetric inequality] \label{th-isoper}
If $\Delta_1$ and $\Delta_2$ are convex regions in the plane then
$$\textup{Area}(\Delta_1)\textup{Area}(\Delta_2) \leq
A(\Delta_1, \Delta_2)^2,$$ where $A(\Delta_1, \Delta_2)$ is the mixed area.
\end{Th}

When $\Delta_2$ is the unit disc in the plane, $A(\Delta_1,
\Delta_2)$ is $1/2$ times the perimeter of $\Delta_1$. Thus the
classical form (\ref{equ-isoper}) of the inequality (for convex
regions) follows from Theorem \ref{th-isoper}.

\begin{proof}[Proof of Theorem \ref{th-isoper}]
It is easy to verify that the isoperimetric inequality is equivalent to
the Brunn-Minkowski for n=2. Let us check this in one direction, i.e.
the isoperimetric inequality follows from  Brunn-Minkowski for $n=2$:
\begin{eqnarray*}
\Area(\Delta_1) +2A(\Delta_1, \Delta_2) +\Area(\Delta_2)
&=& \Area(\Delta_1 + \Delta_2) \cr
&\geq& (\Area^{1/2}(\Delta_1)+\Area^{1/2}(\Delta_2))^2 \cr
&=& \Area(\Delta_1)+2\Area(\Delta_1)^{1/2}\Area(\Delta_2)^{1/2} \cr &+& \Area(\Delta_2), \cr
\end{eqnarray*}
which readily implies the isoperimetric inequality.
\end{proof}

The following generalization of Brunn-Minkowski inequality
is a corollary of Alexandrov-Fenchel inequality.
\begin{Cor}  (Generalized Brunn-Minkowski inequality)
For any $0<m\leq n$ and for any fixed convex bodies
$\Delta_{m+1},\dots, \Delta_n$, the function $F$ which assigns to a
body $\Delta$, the number $F(\Delta) =
V^{1/m}(m*\Delta,\Delta_{m+1},\dots,\Delta_n),$ is concave, i.e. for
any two convex bodies $\Delta_1, \Delta_2$ we have
$$F(\Delta_1)+F(\Delta_2)\leq F(\Delta_1+\Delta_2).$$
\end{Cor}

On the other hand, the usual proof of Alexandrov-Fenchel inequality
deduces it from the Brunn-Minkowski inequality.
But this deduction is the main part (and
the most complicated part) in the proof (\cite{Burago-Zalgaller}).
Interestingly, The main construction in the present paper (using algebraic geometry)
allows us to obtain Alexandrov-Fenchel inequality as an immediate corollary of the
simplest case of the Brunn-Minkowski, i.e. isoperimetric inequality.

\section {Generic systems of Laurent polynomial equations in
$(\c^*)^n$} \label{subsec-B-K}
In this section we recall the
famous results due to Ku\v{s}nirenko and Bernstein on the number of
solutions of a generic system of polynomials in $(\c^*)^n$.

Let us identify the lattice $\z^n$ with {\it Laurent monomials} in
$(\c^*)^n$: to each integral point $k\in \z^n$,
$k=(k_1,\dots,k_n)$ we associate the monomial $z^k=z_1^{k_1}\dots
z_n^{k_n}$ where $z=(z_1, \ldots, z_n)$.
A {\it Laurent polynomial} $P=\sum_k c_k z^k$ is a finite
linear combination of Laurent monomials with complex coefficients.
The {\it support} $\supp(P)$ of a Laurent polynomial $P$, is the set
of exponents $k$ for which $c_k \neq 0$. We denote the convex hull
of a finite set $A\subset \z^n$ by $\Delta_A \subset \r^n$.
The {\it Newton polytope} $\Delta (P)$ of a Laurent polynomial $P$
is the convex hull $\Delta_{\supp(P)}$ of its support. With each
finite set $A\subset \z^n$ one associates a vector space
$L_A$ of Laurent polynomials $P$ with $\supp(P) \subset A$.

\begin{Def}
We say that a property holds for a {\it generic element} of a vector
space $L$ if there is a proper algebraic set $\Sigma$ such that the
property holds for all elements in $L \setminus \Sigma$.
\end{Def}

\begin{Def}  For a given $n$-tuple of finite sets
$A_1,\dots,A_n \subset \z^n$  the {\it intersection index} of
the $n$-tuple of spaces  $[L_{A_1}\dots, L_{A_n}]$ is the number of
solutions in $(\c^*)^n $ of a generic system of equations
$P_1=\dots=P_n=0$, where $P_1\in L_1,\dots, P_n\in L_n$.
\end {Def}

\noindent {\bf Problem:} {\it Find the intersection index
$[L_{A_1}\dots, L_{A_n}]$}, that is, for a generic element $(P_1,
\ldots, P_n) \in L_{A_1} \times\dots\times L_{A_n}$ find a formula
for the number of solutions in $(\c^*)^n $ of the system of
equations $P_1=\dots=P_n=0$.

Ku\v{s}nirenko found the following important result which answers a particular case of the above problem
\cite{Kushnirenko}:

\begin{Th} \label{th-Kush} When the convex hulls of the sets $A_i$ are the same
and equal to a polytope $\Delta$ we have $$[L_{A_1}, \dots, L_{A_n}]=n! \Vol(\Delta),$$
where $\Vol$ is the standard $n$-dimensional volume in $\r^n$.
\end{Th}

According to Theorem \ref{th-Kush} if $P_1, \dots, P_n$ are
sufficiently general Laurent polynomials with given Newton polytope
$\Delta$, the number of solutions in $(\c^*)^n $ of the system
$P_1=\dots=P_n=0$ is equal to $n!\Vol(\Delta)$.

The problem was solved by Bernstein in full generality \cite{Bernstein}:
\begin{Th} \label{th-Bern} In the general case, i.e. for arbitrary
finite subsets $A_1, \ldots, A_n \subset \z^n$ we have
$$[L_{A_1}, \dots, L_{A_n}]=n!
V(\Delta_{A_1}, \ldots, \Delta_{A_n}),$$ where $V$ is the mixed
volume of convex bodies in $\r^n$.
\end{Th}

According to Theorem \ref{th-Bern} if  $P_1, \dots,P_n$ are
sufficiently general Laurent polynomials with Newton polyhedra
$\Delta_1, \ldots, \Delta_n$ respectively, the number of solutions
in $(\c^*)^n $ of the system $P_1=\dots=P_n=0$ is equal to
$n!V(\Delta_1, \ldots, \Delta_n)$.

\section{Convex Geometry and  Bernstein-Ku\v{s}nirenko theorem}
\label{subsec-B-K-convex-geo}
Let us examine Theorem \ref {th-Bern} (which we will call
Bernstein-Ku\v{s}nirenko theorem) more closely.
In the space of regular functions
on $(\c^*)^n$ there is a natural family of finite dimensional
subspaces, namely the subspaces which are stable under the action of
the multiplicative group $(\c^*)^n$. Each such subspace is of the
form $L_A$ for some finite set $A\subset \z^n$ of monomials.

For two finite dimensional subspaces $L_1$,$L_2$ of regular
functions in $(\c^*)^n $, let us define the product $L_1L_2$ as the
subspace spanned by the products $fg$, where $f\in L_1$, $g\in L_2$.
Clearly multiplication of monomials corresponds to the addition of
their exponents, i.e. $z^{k_1}z^{k_2}=z^{k_1+k_2}.$ This implies
that $L_{A_1}L_{A_2}=L_{A_1+A_2}$.

The Bernstein-Ku\v{s}nirenko theorem defines and computes the
intersection index $[L_{A_1}, L_{A_2}, \ldots, L_{A_n}]$ of the
$n$-tuples of subspaces  $L_{A_i}$ for finite subsets $A_i \subset
\z^n$. Since this intersection index is equal to the mixed volume,
it enjoys the same properties, namely: 1) Positivity; 2)
Monotonicity; 3) Multi-linearity; and 4) Alexandrov-Fenchel inequality
and its corollaries. Moreover, if for a finite set $A \subset \z^n$ we
let $\overline{A} = \Delta_A \cap
\z^n$, then 5) the spaces $L_A$ and $L_{\overline{A}}$ have the same intersection indices.
That is, for any $(n-1)$-tuple of finite subsets $A_2,
\ldots A_n \in \z^n$,
$$[L_A, L_{A_2}, \ldots, L_{A_n}] = [L_{\overline{A}}, L_{A_2}, \ldots, L_{A_n}].$$
This means that (surprisingly!) enlarging $L_A \mapsto L_{\overline{A}}$
does not change any of the intersection indices we considered. And hence
in counting the number of solutions of a system, instead of
support of a polynomial, its convex hull plays the main role. Let us
denote the subspace $L_{\overline{A}}$ by $\overline{L_A}$ and call
it the {\it completion of $L_A$}.

Since the semi-group of convex bodies with Minkowski sum has cancelation
property, the following cancelation property
for the finite subsets of $\z^n$ holds: if for finite subsets $A, B, C \in
\z^n$ we have $\overline{\overline{A} + \overline{C}} =
\overline{\overline{B} + \overline{C}}$ then $\overline{A} =
\overline{B}$. And we have the same cancelation property for the
corresponding semi-group of subspaces $L_A$. That is, if
$\overline{\overline{L_A} ~\overline{L_C}} =
\overline{\overline{L_B}~\overline{L_C}}$ then
$\overline{L_A} = \overline{L_B}$.


Bernstein-Ku\v{s}nirenko theorem relates mixed volume in convex geometry with
intersection index in algebraic geometry. In algebraic geometry
the following inequality about intersection indices
on a surface is well-known.
\begin{Th}[Hodge inequality]\label{th-Hodge} Let $\Gamma_1, \Gamma_2$ be
algebraic curves on a smooth irreducible projective surface. Assume
that $\Gamma_1, \Gamma_2$ have positive self-intersection indices.
Then
$$ (\Gamma_1, \Gamma_2)^2\geq (\Gamma_1, \Gamma_1)(\Gamma_2,
\Gamma_2)$$ where $(\Gamma_i,\Gamma_j)$ denotes the intersection index of
the curves $\Gamma_i$ and $\Gamma_j$.
\end{Th}

On one hand Theorem \ref {th-Bern} allows one to prove
Alexandrov-Fenchel inequality algebraically using Theorem \ref{th-Hodge}
(see \cite{Khovanskii-BZ, Teissier}). On the other hand Hodge inequality suggests
an analogy between the mixed volume theory and the intersection theory
of Cartier divisors on a projective algebraic variety.

We will return back to this discussion after statement of our main
theorem (Theorem \ref{th-main}) and its corollary which is a version of
Hodge inequality.

\section{Analog of the Intersection theory of Cartier divisors for
non-complete varieties }
Now we discuss general results inspired by
Bernstein-Ku\v{s}nirenko theorem which can be considered as an
analogue of the intersection theory of Cartier divisors for
non-complete varieties (\cite{Askold-Kiumars-arXiv2}). Instead of
$(\c^*)^n $ we take any
irreducible $n$-dimensional quasi-projective variety $X$ and instead of
a finite dimensional space of functions spanned by monomials we take
any finite dimensional space of rational functions. For these spaces
we define an intersection index and prove that it enjoys
all the properties of the mixed volume of convex bodies.

Consider the collection ${\bf K}_{rat}(X)$ of all finite dimensional
subspaces of rational functions on $X$. The set ${\bf K}_{rat}(X)$
has a natural multiplication: product $L_1L_2$ of two subspaces
$L_1, L_2 \in {\bf K}_{rat}(X)$ is the subspace spanned by all the
products $fg$ where $f \in L_1$, $g \in L_2$. With respect to this
multiplication, ${\bf K}_{rat}(X)$ is a commutative semi-group.
\begin{Def} The {\it intersection index} $[L_1,\dots, L_n]$  of
$L_1,\dots,L_n \in {\bf K}_{rat}(X)$ is the number of solutions in $
X$ of a generic system of equations $f_1=\dots=f_n=0$, where
$f_1\in L_1,\dots, f_n\in L_n$. In counting the solutions, we
neglect solutions $x$ at which all the functions in some space $L_i$
vanish as well as solutions at which at least one function from some
space $L_i$ has a pole.
\end {Def}
More precisely, let $\Sigma \subset X$ be a hypersurface which
contains: 1) all the singular points of $X$; 2) all the poles of
functions from any of the $L_i$; 3) for any $i$, the set of common
zeros of all the $f \in L_i$. Then for a generic choice of $(f_1,
\ldots, f_n) \in L_1 \times \cdots \times L_n$, the intersection
index $[L_1, \ldots, L_n]$ is equal to the number of solutions $\{x
\in X \setminus \Sigma \mid f_1(x) = \ldots = f_n(x) = 0 \}$.

\begin{Th} The intersection index $[L_1, \ldots, L_n]$ is well-defined. That is, there is
a Zariski open subset $U$ in the vector space $L_1 \times \cdots \times L_n$ such that for any
$(f_1, \ldots, f_n) \in U$ the number of solutions $x \in X \setminus \Sigma$
of the system $f_1(x) = \ldots = f_n(x) = 0$ is the same (and hence equal to $[L_1, \ldots, L_n]$).
Moreover the above number of solutions is independent of the choice of $\Sigma$ containing 1)-3) above.
\end{Th}

The following properties of the intersection index are easy consequences of the definition:
\begin{Prop} \label{prop-obvious-prop-int-index}
1) $[L_1,\dots,L_n]$ is a symmetric function of
the n-tuples $L_1,\dots,L_n \in {\bf K}_{rat}(X)$, (i.e. takes the same value under a
permutation of $L_1,\dots,L_n$); 2) The intersection index
is monotone, (i.e. if $L'_1\subseteq L_1,\dots, L'_n\subseteq L_n$,
then $[L_1,\dots,L_n] \geq [L'_1,\dots,L'_n]$; and 3) The intersection index is non-negative
(i.e. $[L_1,\dots,L_n] \geq 0$).
\end{Prop}

The next two theorems contain the main properties of the intersection index.
\begin{Th}[Multi-linearity] \label{th-multi-lin-int-index}
1) Let $L_1', L_1'', L_2, \ldots, L_n \in {\bf K}_{rat}(X)$
and put $L_1= L_1'L_1''$. Then
$$[L_1,\dots,L_n]=[L'_1,\dots,L_n]+[L''_1,\dots,L_n].$$
2) Let $L_1, \ldots, L_n \in {\bf K}_{rat}(X)$ and take $1$-dimensional subspaces
$L'_1, \ldots, L'_n \in {\bf K}_{rat}(X)$. Then
$$[ L_1, \ldots, L_n ] = [ L'_1L_1, \ldots, L'_nL_n].$$
\end{Th}

Let us say that $f \in \c(X)$ is {\it integral over a subspace} $L \in {\bf K}_{rat}(X)$ if
$f$ satisfies an equation
$$f^m+a_1f^{m-1}+\dots +a_m=0,$$
where  $m>0$ and $a_i \in L^i$, for each $i=1,\ldots, m$. It is
well-known that the collection $\overline{L}$ of all integral
elements over $L$ is a vector subspace containing $L$. Moreover if
$L$ is finite dimensional then $\overline{L}$ is also finite
dimensional (see \cite[Appendix 4]{Zariski}). It is called the {\it
completion of $L$}. For two subspaces $L, M \in \K$ we say that $L$
is equivalent to $M$ (written $L \sim M$) if there is $N \in \K$
with $LN = MN$. One shows that the completion $\overline{L}$ is in
fact the largest subspace in $\K$ which is equivalent to $L$. The
enlarging $L\rightarrow \bar L$ is analogous to the geometric
operation $A \mapsto \Delta (A)$ which associates to a finite set
$A$ its convex hull $\Delta(A)$.

\begin{Th} \label{th-int-index-completion}
1) Let $L_1 \in K_{rat}(X)$ and let
$G_1\in K_{rat}(X)$ be the subspace spanned by $L_1$ and
a rational function $g$ integral over $L_1$. Then for any
$(n-1)$-tuple $L_2,\dots,L_n\in K_{rat}(X)$ we have
$$[L_1,L_2,\dots,L_n]=[G_1,L_2,\dots,L_n].$$
2) Let $L_1 \in {\bf K}_{rat}(X)$ and let $\overline{L_1}$ be its completion as defined above.
Then for any $(n-1)$-tuple $L_2,\dots,L_n\in {\bf K}_{rat}(X)$ we have
$$[L_1,L_2,\dots,L_n]=[\overline{L_1},L_2,\dots,L_n].$$
\end{Th}

As with any other commutative semi-group, there corresponds a
Grothendieck group to the semi-group ${\bf K}_{rat}(X)$. For a
commutative semi-group $K$, the Grothendieck group $G(K)$ is the
unique abelian group defined with the following universal property:
there is a homomorphism $\phi: K \to G(K)$ and for any abelian group
$G'$ and homomorphism $\phi':K \to G'$, there exist a homomorphism
$\psi: G(K) \to G'$ such that $\phi' = \psi \circ \phi$. The
Grothendieck group can be defined constructively also: for $x, y \in
K$ we say $x \sim y$ if there is $z \in K$ with $xz=yz$. Then $G(K)$
is the group of formal quotients of equivalence classes of $\sim$.
From the multi-linearity of the intersection index it follows that
the intersection index extends to the the Grothendieck group of ${\bf
K}_{rat}(X)$.

The Grothendieck group of ${\bf K}_{rat}(X)$ can be considered as
an analogue (for a non-complete variety $X$) of the group of Cartier
divisors on a projective variety, and the intersection index on
the Grothendieck group of ${\bf K}_{rat}(X)$ as an analogue of the intersection
index of Cartier divisors.

The intersection theory on the Grothendieck group of ${\bf K}_{rat}(X)$
enjoys all the properties of mixed volume. Some of such
properties we already discussed in the present section. The others
will be discussed later (see Theorem \ref {th-Alex-Fenchel-alg} and
Corollary \ref {cor-Alex-Fenchel-alg} below).

\section{Proof of Bernstein-Ku\v{s}nirenko theorem via Hilbert
theorem}\label {sec-proof-B-K} Let us recall the proof of the
Bernstein-Ku\v{s}hnirenko theorem from \cite{Khovanskii-finite-sets}
which will be important for our generalization.

For each space $L \in \K$ let us define the Hilbert function $H_L$
by $H_L(k) = \dim(L^k)$. For sufficiently large values of $k$, the function $H_L(k)$ is a
polynomial in $k$, called the {\it Hilbert polynomial} of $L$.

With each space $L \in \K$, one associates a rational {\it Kodaira
map} from $X$ to $\p(L^*)$, the projectivization of the dual space
$L^*$: to any $x \in X$ there corresponds a functional in $L^*$
which evaluates $f \in L$ at $x$. The Kodaira map sends $x$ to the
image of this functional in $\p(L^*)$. It is a rational map, i.e.
defined on  a Zariski open subset in $X$. We denote by $Y_L$ the
subvariety in the projective space $\p(L^*)$ which is equal to the
closure of the image of $X$ under the Kodaira map.

The following theorem is a version of the classical Hilbert theorem
on the degree of a subvariety of the projective space.
\begin{Th}[Hilbert's theorem]\label{th-Hil}
The degree  of the Hilbert
polynomial of the space $L$  is equal to the dimension of the
variety $Y_L$, and its leading coefficient $c$ is the degree of $Y_L
\subset \p(L^*)$ divided by $m!$.
\end{Th}

Let $A$ be a finite subset in $\z^n$ with $\Delta(A)$ its convex hull.
Denote by $k*A$ the sum $A+\dots+A$ of $k$ copies of the set $A$,
and by $(k\Delta (A))_C$ the subset  in
$k\Delta (A)$ containing points whose distance to the boundary
$\partial(k\Delta (A))$ is bigger than $C$. The following combinatorial theorem
gives an estimate for the set $k * A$ in terms of the set of integral points in $k\Delta(A)$.

\begin{Th}[\cite{Khovanskii-finite-sets}] \label{th-A}
1) One has  $k*A\subset k\Delta (A)\cap \z^n$.
2) Assume that the differences $a-b$ for $a, b \in A$ generate the
group $\z^n$. Then there exists a constant $C$ such that for any $k \in \n$ we have
$$ (k\Delta (A))_C\cap \z^n\subset k*A.$$
\end{Th}

\begin{Cor} \label{cor-vol-Delta(A)}
Let $A \subset \z^n$ be a finite subset satisfying the condition in Theorem \ref {th-A}(2).
Then $$\lim_{k \to \infty} \frac{\#(k*A)}{k^n} = \Vol_n(\Delta(A)).$$
\end{Cor}

Corollary \ref {cor-vol-Delta(A)} together with the Hilbert theorem (Theorem
\ref{th-Hil}) proves the Ku\v{s}nirenko theorem for sets $A$ such
that the differences $a-b$ for $a, b \in A$ generate the group
$\z^n$. The Ku\v{s}nirenko theorem for the general case
easily follows from this. The Bernstein theorem \ref {th-Bern} follows from the
Ku\v{s}nirenko theorem \ref {th-Kush} and the identity
$L_{A+B}=L_AL_B$.

\section{Graded semigroups in $\n \oplus \z^n$ and Newton convex body} \label{sec-semi-gp}
Let $S$ be a subsemi-group of $\n \oplus \z^n$. For any integer $k > 0$ we denote by $S_k$
the section of $S$ at level $k$, i.e. the set of elements $x \in \z^n$ such that $(k, x) \in S$.

\begin{Def} \label{def-graded-semi-gp}
A subsemi-group $S$ of $\n \oplus \z^n$ is called:

1) a {\it graded semi-group} if for any $k>0$, $S_k$ is finite and
non-empty;

2) an {\it ample semi-group} if there is a natural $m$ such that the
differences $a-b$ for $a, b \in S_m$ generate the group $\z^n$;

3) a {\it semi-group with restricted growth} if there is constant
$C$ such that for any $k>0$ we have $\#(S_k) \leq Ck^n$.
\end{Def}

For a graded semi-group $S$, let $\Con(S)$ denote the convex hull of
$S \cup \{0\}$. It is a cone in $\r^{n+1}$. Denote by $M_S$ the
semi-group $\Con(S) \cap (\n \oplus \z^n)$. The semigroup $M_S$
contains the semigroup $S$.

\begin{Def} \label{def-Newton-conv-set-semi-gp}
For a graded semi-group $S$, define the {\it Newton convex set}
$\Delta(S)$ to be the section of the cone $\Con(S)$ at $k=1$, i.e.
$$\Delta(S) = \{x \mid (1, x) \in \Con(S) \}.$$
\end{Def}

\begin{Th}[Asymptotic of  graded semi-groups] \label{th-asym-graded-semi-gp}
Let $S$ be an ample graded semi-group with restricted growth in $\n
\oplus \z^n$.  Then:

1) the cone $Con(S)$ is strictly convex, i.e. the Newton convex set
$\Delta(S)$  is bounded;

2) Let $d(k)$ denote the maximum distance of the points $(k, x)$ from the boundary of
$\Con(S)$ for $x \in M_S(k) \setminus S(k)$. Then
$$\lim_{k \to \infty} \frac{d(k)}{k} = 0.$$
\end{Th}

Theorem \ref{th-asym-graded-semi-gp} basically follows from
Theorem \ref{th-A}. For the sketch of its proof see \cite{Askold-Kiumars-arXiv}.

\begin{Cor} \label{cor-Sk-vol-Delta}
Let $S$ be an ample graded semi-group with restricted growth in $\n
\oplus \z^n$. Then
$$\lim_{k \to \infty} \frac{\#(S_k)}{k^n} = \Vol_n(\Delta(S)).$$
\end{Cor}

\section{Valuations on the field of rational functions}
We start with the definition of a pre-valuation. Let $V$ be a vector space and let $I$ be a set
totally ordered with respect to some ordering $<$.

\begin{Def}
A {\it pre-valuation}
on $V$ with values in $I$ is a function $v: V\setminus \{0\} \to I$ satisfying the following:
1) For all $f,g \in V$, $v(f+g) \geq \min(v(f), v(g))$;
2) For all $f \in V$ and $\lambda \neq 0$, $v(\lambda f) = v(f)$;
3) If for $f, g \in V$ we have $v(f) = v(g)$ then there is $\lambda \neq 0$ such that
$v(g - \lambda f) > v(g)$.
\end{Def}

It is easy to verify that if $L \subset V$ is a finite dimensional subspace then
$\dim(L)$ is equal to $\#(v(L))$.

\begin{Ex} \label{ex-pre-valuation}
Let $V$ be a finite dimensional vector space with basis $\{e_1, \ldots, e_n\}$ and let
$I = \{1, \ldots, n\}$ with usual ordering of numbers. For $f = \sum_i \lambda_i e_i$ define
$$v(f) = \min\{ i \mid \lambda_i \neq 0\}.$$
\end{Ex}

\begin{Ex} [Schubert cells in the Grassmannian] \label{ex-Grassmannian}
Let $\textup{Gr}(n, k)$ be the Grassmannian of $k$-dimensional
planes in $\c^n$.  In Example \ref{ex-pre-valuation} take $V=\c^n$
with standard basis. Under the pre-valuation $v$ above each
$k$-dimensional subspace $L \subset \c^n$ goes to a subset $M\subset
I$ containing $k$ elements. The set of all $k$-dimensional subspaces
which are mapped onto $M$ forms the {\it Schubert cell $X_M$} in the
Grassmannian $\textup{Gr}(n, k)$.
\end{Ex}

In a similar fashion to Example \ref {ex-Grassmannian}  the Schubert
cells in the variety of complete flags can also be recovered from
the pre-valuation $v$ above on $\c^n$.

Next we define the notion of a valuation with values in a totally ordered abelian group.

\begin{Def}
Let $K$ be a field and $\Gamma$ a totally ordered abelian group.
A pre-valuation $v: K\setminus \{0\} \to \Gamma$ is a {\it valuation} if moreover it satisfies the
following: for any $f, g \in K$ with $f, g \neq 0$, we have $$v(fg) = v(f) + v(g).$$
The valuation $v$ is called {\it faithful} if its image is the whole $\Gamma$.
\end{Def}

We will only be concerned with the field $\c(X)$ of rational
functions on an $n$-dimensional irreducible variety $X$, and
$\z^n$-valued valuations on it (with respect to some total order on
$\z^n$).

\begin{Ex} \label{ex-valuation-curve}
Let $X$ be an irreducible  curve. Take the field of rational
functions $\c(X)$ and $\Gamma = \z$. Take a smooth point $a$ on $X$.
Then the map
$$v(f) = \ord_a(f)$$ defines a faithful valuation on $\c(X)$.
\end{Ex}

\begin{Ex} \label{ex-valuation-Grobner}
Let $X$ be an irreducible $n$-dimensional variety. Take a smooth
point $a \in X$. Consider a local system of coordinates with
analytic coordinate functions $x_1, \ldots, x_n$ and with the origin
at the point $a$. Let $\Gamma = \z^n_+$ be the semigroup in $\z^n$
of points with non-negative coordinates. Take any well-ordering
$\prec$ which respects the addition, i.e. if $a\prec b$ then
$a+c\prec b+c$.  For a germ $f$ at the point $a$ of an analytic
function in $x_1, \ldots, x_n$ let $c x^{\bold \alpha (f)}=
cx_1^{\alpha_1} \cdots x_n^{\alpha_n}$ be the term in the Taylor
expansion of $f$ with minimum exponent $\alpha (f) =(\alpha_1,
\ldots, \alpha_n)$, with respect to the ordering $\prec$. For a germ
$F$ at the point $a$ of a meromorphic function $F=f/g$  define
$v(F)$ as $\alpha (f)-\alpha (g)$. This function $v$ induces a
faithful valuation on the field of rational functions $\c(X)$.
\end{Ex}

\begin {Ex} Let $X$ be an irreducible $n$-dimensional variety and
$Y$ any variety birationally isomorphic to $X$. Then
fields $\c(X)$ and $\c(Y)$ are isomorphic and thus any faithful
valuation on $\c(Y)$ gives a faithful valuation on $\c(X)$ as well.
\end {Ex}


\section{Main construction and theorem} \label{sec-main}
Let $X$ be an irreducible $n$-dimensional variety. Fix a faithful valuation
$v: \c(X) \to \z^n$, where $\z^n$ is equipped with any total ordering respecting addition.

Let $L \in \K$ be a finite dimensional subspace of rational functions.
Consider the semi-group $S(L)$ in $\n \oplus \z^n$ defined by
$$S(L) = \bigcup_{k>0} \{(k, v(f)) \mid f \in L^k\}.$$
It is easy to see that $S(L)$ is a graded semi-group. Moreover by
the Hilbert theorem $S(L)$ is contained in a semi-group of
restricted growth.

\begin{Def}[Newton convex body for a subspace of rational functions]
We define the {\it Newton convex body} for a subspace $L$ to be the
convex body $\Delta(S(L))$ associated to the semi-group $S(L)$.
\end{Def}

Denote by $s(L)$  the index of the subgroup in $\z^{n}$ generated by
all the differences $a-b$ such that $a$, $b$
belong to the same set $S_m(L)$ for some $m>0$. Also let $Y_L$
be the closure of the image of the variety $X$ (in fact image of a Zariski open
subset of $X$) under the Kodaira rational map
$\Phi_L: X \to \p(L^*)$. If $\dim(Y_L)$ is equal to $\dim(X)$ then the Kodaira map from $X$ to
$Y_L$ has finite mapping degree. Denote this mapping degree by
$d(L)$.

\begin{Th}[Main theorem] \label{th-main} Let $X$ be an irreducible
$n$-dimensional \qp variety and let $L\in\K$ with the Kodaira map $\Phi_L: X \to
\p(L^*)$. Then:

1) Complex dimension of the variety $Y_L$ is equal to the real
dimension of the Newton convex body $\Delta(S(L))$.

2) If $\dim(Y_L)=n$ then
$$[L, \ldots, L] = \frac{n!d(L)}{s(L)} \Vol_n(\Delta(S(L))).$$

3) In particular,  if $\Phi_L: X\rightarrow Y_L $ is a birational
isomorphism then
$$[L,\dots,L]= n! \Vol_n(\Delta(S(L))).$$

4) For any two subspaces $L_1, L_2 \in \K$ we have
$$\Delta(S(L_1))+\Delta(S(L_2)) \subseteq \Delta(S(L_1L_2)).$$
\end{Th}

The proof of the main theorem is based on Theorem
\ref{th-asym-graded-semi-gp} (which describes the asymptotic behavior of
an ample graded semigroup of restricted growth) and the Hilbert
(Theorem \ref {th-Hil}). The sketch of proof can be found in \cite{Askold-Kiumars-arXiv}.

\section{Algebraic analogue of Alexandrov-Fenchel  inequalities} \label{sec-Alexandrov-Fenchel}
Part (2) of the main theorem (Theorem \ref {th-main}) can be considered
as a wide-reaching generalization of the Ku\v{s}nirenko theorem,
in which instead of $(\c^*)^n$ one takes any
$n$-dimensional irreducible variety $X$, and instead of a finite
dimensional space generated by monomial one takes any finite
dimensional space $L$ of rational functions.
The proof of Theorem \ref{th-main} is an extension
of the arguments used in \cite{Khovanskii-finite-sets}
to prove Ku\v{s}nirenko theorem (see also Section \ref {sec-proof-B-K}).
As we  mentioned the Bernstein theorem (Theorem \ref {th-Bern}) follows
immediately from the Ku\v{s}nirenko theorem and the identity
$$L_{A+B}=L_AL_B.$$ Thus the Bernstein-Ku\v{s}nirenko theorem is
a corollary of our Theorem \ref{th-main}.

Note that although the Newton convex body
$\Delta(S(L))$ depends on a choice of a faithful valuation, its volume
depends on $L$ only: after multiplication by $n!$ it equals
the self-intersection index $[L,\dots, L]$.

Our generalization of the Ku\v{s}hnirenko theorem does not imply the
generalization of the Bernstein theorem. The point is that
in general we do not always have an equality $\Delta(S(L_1))+\Delta(S(L_2)) =
\Delta(S(L_1L_2))$. In fact, by Theorem \ref{th-main}(4), what is always true is the inclusion
$$\Delta(S(L_1))+\Delta(S(L_2)) \subseteq \Delta(S(L_1L_2)).$$
This inclusion is sufficient for us to prove the following
interesting corollary:

Let us call a subspace $L \in \K$ a {\it big subspace} if for some $m>0$ the Kodaira rational map
of $L^m$ is a birational isomorphism between $X$ and its image. It is not hard to show that
the product of two big subspaces is again a big subspace and thus the big subspaces form a subsemi-group
of $\K$.
\begin{Cor}[Algebraic analogue of Brunn-Minkowski] \label{cor-Brunn-Mink-int-index}
Assume that $L,G\in \K$ are big subspaces.
Then $$[L,\dots, L]^{1/n} +[G,\dots, G]^{1/n}\leq [LG,\dots,
LG]^{1/n}.$$
\end{Cor}
\begin{proof} Since replacing $L$ and $G$ by $L^m$ and $G^m$ does not change the inequality, without loss of generality,
we can assume that the Kodaira maps of $L$ and $G$ are birational isomorphisms onto their images.
From Part (4) in Theorem \ref{th-main} we
have $\Delta (S(L))+\Delta (S(G))\subseteq \Delta (S(LG))$. So $\Vol
(\Delta (S(L))+\Delta (S(G)))\leq \Vol(\Delta (S(LG)))$. Also from Part (3) in
the same theorem we have $$[L,\dots,L]= n!
\Vol(\Delta(S(L)),$$ $$[G,\dots,G]= n! \Vol(\Delta(S(G)),$$
$$[LG,\dots,LG]= n! \Vol(\Delta(S(LG)).$$ To complete the proof
it is enough to use the Brunn-Minkowski inequality.
\end{proof}

\begin{Cor}[A version of Hodge inequality] \label{cor-Hodge} . If $L,G\in \K$ are big subspaces
and $X$ is an algebraic surface then
$$[L,L][G,G]\leq [L,G]^2.$$
\end{Cor}
\begin{proof}
From Corollary \ref{cor-Brunn-Mink-int-index}, for $n=2$, we have
\begin{eqnarray*}
[L,L] +2[L,G] +[G,G] &=& [LG,LG] \cr &\geq& ([L,L]^{1/2}
+[G,G]^{1/2})^2  \cr &=& [L,L]+2[L,L]^{1/2}[G,G]^{1/2}+[G,G], \cr
\end{eqnarray*}
which readily implies Hodge inequality.
\end{proof}

Thus Theorem \ref{th-main} immediately enables us to reduce the Hodge
inequality to the isoperimetric inequality. This way, we can
easily prove an analogue of Alexandrov-Fenchel inequality and its corollaries
for intersection index:

\begin{Th}[Algebraic analogue of Alexandrov-Fenchel
inequality] \label{th-Alex-Fenchel-alg} Let $X$ be an irreducible
$n$-dimensional \qp variety and let $L_1,\dots,L_n\in \K$ be big subspaces.
Then the following inequality holds
$$[L_1,L_2, L_3,\dots,L_n]^2\geq [L_1,L_1,
L_3,\dots,L_n][L_2,L_2,L_3,\dots,L_n].$$
\end{Th}

\begin{Cor}[Corollaries of the algebraic analogue of
Alexandrov--Fenchel inequality] \label{cor-Alex-Fenchel-alg} Let $X$
be an $n$-dimensional irreducible \qp variety.

1) Let $2\leq m\leq n$ and $k_1+\dots+k_r=m$ with $k_i \in
\n$. Take big subspaces of rational functions $L_1, \ldots, L_n
\in \K$. Then $$[k_1*L_1, \ldots, k_r*L_r, L_{m+1},
\ldots, L_n]^m \geq \prod_{1 \leq j \leq r}[m*L_j, L_{m+1}, \ldots,
L_n]^{k_j}.$$

2)(Generalized Brunn-Minkowski inequality) For any fixed
big subspaces $L_{m+1}, \ldots, L_n \in \K$, the
function $$F: L \mapsto [m*L, L_{m+1}, \ldots, L_n]^{1/m},$$ is a
concave function on the semi-group of big subspaces.
\end{Cor}

As we saw above, Bernstein-Ku\v{s}nirenko theorem follows from the
main theorem. Applying algebraic analogue of Alexandrov-Fenchel
inequality to the situation considered in Bernstein-Ku\v{s}nirenko theorem
one can  prove  Alexandrov-Fenchel inequality for convex
polyhedra with integral vertices. By homogeneity it implies
Alexandrov-Fenchel inequality for convex polyhedra with rational
vertices. But since each convex body can be approximated by polyhedra with
rational vertices, by continuity we obtain a proof of
Alexandrov-Fenchel inequality in complete generality.

Thus Bernstein-Ku\v{s}nirenko theorem and Alexandrov-Fenchel
inequality in algebra and in geometry can be considered as
corollaries of our Theorem \ref{th-main}.

\section{Additivity of the Newton convex body for varieties with reductive group action}
While the additivity of the Newton convex body does not hold
in general but, as mentioned in Section \ref{sec-proof-B-K}, it holds for the subspaces $L_A$ of
Laurent polynomials on $(\c^*)^n$ spanned by monomials. The subspaces $L_A$ are exactly the subspaces which are stable under the natural action of the multiplicative group $(\c^*)^n$ on Laurent polynomials (induced by the natural action of $(\c^*)^n$ on itself).
We will see that the additivity generalizes to some classes of varieties with a reductive group action.

Let $G$ be a connected reductive algebraic group over $\c$, i.e. the complexification of a
connected compact real Lie group. Also let $X$ be a $G$-variety, that is a variety equipped with an algebraic
action of $G$.

The group $G$ naturally acts on $\c(X)$ by $(g \cdot f)(x) = f(g^{-1}\cdot x)$.
A subspace $L \in \K$ is $G$-stable if for any $f \in L$ and $g \in G$ we have $g \cdot f \in L$.

\begin{Th} \label{th-G-var-valuation}
Let $X$ be an $n$-dimensional variety with an algebraic action of $G$. Then
there is a naturally defined faithful valuation $v: \c(X) \to \z^n$ such that
for any $G$-stable subspace $L \in \K$, the Newton convex body $\Delta(S(L))$ is
in fact a polytope.
\end{Th}

\begin{Def}
Let $V$ be a finite dimensional representation of $G$. Let $v = v_1 + \ldots + v_k$ be a
sum of highest weight vectors in $V$. The closure of the $G$-orbit of $v$ in $V$ is called an {\it S-variety}.
\end{Def}

Affine toric varieties are $S$-varieties for $G=(\c^*)^n$.

\begin{Th} \label{th-Newton-polytope-additive}
Let $X$ be an $S$-variety for one of the groups $G=\SL(n, \c)$, $\SO(n, \c)$, $\SP(2n, \c)$, $(\c^*)^n$,
or a direct product of them. Then for the valuation in Theorem \ref{th-G-var-valuation}
and for any choice of $G$-stable subspaces $L_1, L_2$ in $\K$ we have
$$\Delta(S(L_1L_2)) = \Delta(S(L_1)) + \Delta(S(L_2)).$$
\end{Th}

\begin{Cor}[Bernstein theorem for $S$-varieties] \label{cor-Bernstein-S-var}
Let $X$ be an $S$-variety for one of the groups $G=\SL(n, \c)$, $\SO(n, \c)$, $\SP(2n, \c)$, $(\c^*)^n$,
or a direct product of them. Let $L_1, \ldots, L_n \in \K$ be
$G$-stable subspaces. Then, for the valuation in Theorem \ref{th-G-var-valuation}, we have
$$[L_1, \ldots, L_n] = n!V(\Delta(S(L_1)), \ldots, \Delta(S(L_n))),$$ where
$V$ is the mixed volume.
\end{Cor}

Another class of $G$-varieties for which the additivity of the Newton polytope holds is the
class of symmetric homogeneous spaces.
\begin{Def}
Let $\sigma$ be an involution of $G$, i.e. an order $2$ algebraic automorphism. Let $H = G^\sigma$ be the
fixed point subgroup of $\sigma$. The homogeneous space $G/H$ is called a {\it symmetric homogeneous space}.
\end{Def}

\begin{Ex}
The map $M \mapsto (M^{-1})^t$ is an involution of $G = \SL(n, \c)$ with
fixed point subgroup $H = \SO(n, \c)$. The symmetric homogeneous space $G/H$ can be identified with the
space of non-degenerate quadrics in $\c P^{n-1}$.
\end{Ex}

Any symmetric homogeneous space is an affine $G$-variety (with left $G$-action).

Under mild conditions on the $L_i$, analogues of
Theorem \ref{th-Newton-polytope-additive} and Corollary \ref{cor-Bernstein-S-var}
hold for symmetric varieties.

Finally, the above theorems extend to subspaces of sections of $G$-line bundles.

\noindent Askold G. Khovanskii\\Department of
Mathematics\\University of Toronto \\Toronto, ON M5S 2E4\\Canada\\
{\it Email:} {\sf askold@math.utoronto.ca}\\

\noindent Kiumars Kaveh\\Department of Mathematics\\University of
Toronto\\Toronto, ON M5S 2E4\\Canada\\
{\it Email:} {\sf kaveh@math.utoronto.ca}\\
\end{document}